\documentclass[11pt,reqno]{amsart}

\usepackage{graphicx}
\usepackage{verbatim}
\usepackage[noadjust]{cite}
\newtheorem{thm}{Theorem}[section]
\newtheorem{theorem}[thm]{Theorem}
\newtheorem{cor}[thm]{Corollary}
\newtheorem{lem}[thm]{Lemma}

\newtheorem{prop}[thm]{Proposition}
\theoremstyle{definition}

\newtheorem*{definition}{Definition}

\theoremstyle{remark}
\newtheorem{rem}{Remark}[section]
\numberwithin{equation}{section}

\def\N{{\mathbb N}}
\def\R{{\mathbb R}}

\newcommand{\us}{u^{\ast}}
\newcommand{\vs}{v^{\ast}}
\newcommand{\ws}{{w^{\ast}}}
\newcommand{\Xs}{X^{\ast}}
\newcommand{\xs}{x^{\ast}}

\newcommand{\Ys}{Y^{\ast}}
\newcommand{\ys}{y^{\ast}}

\newcommand{\Zs}{Z^{\ast}}
\newcommand{\zs}{z^{\ast}}
\newcommand{\eps}{\varepsilon}

\begin{document}

\title[On convex combinations of slices]{On convex combinations of slices of the\\ unit ball in Banach spaces}%
\author{Rainis Haller}%
\author{Paavo Kuuseok}
\author{M\"art P\~oldvere}
\address{Institute of Mathematics and Statistics, University of Tartu,\newline J.~Liivi 2, 50409 Tartu, Estonia}%
\email{\tiny rainis.haller@ut.ee, paavo.kuuseok@gmail.com,
mart.poldvere@ut.ee}%
\thanks{The research was supported by institutional research funding IUT20-57 of the Estonian Ministry of Education and Research.}

\subjclass[2010]{Primary 46B20, 46B22}%
\keywords{Banach space, slice, convex combination of slices, relatively weakly open set}%

\begin{abstract}
We prove that the following three properties for a Banach space are all different from each other: every finite convex combination of slices of the unit ball is (1) relatively weakly open, (2) has nonempty interior in relative weak topology of the unit ball, and (3) intersects the unit sphere. In particular, the $1$-sum of two Banach spaces does not have property (1), but it has property (2) if both the spaces have property (1); the Banach space $C[0,1]$ does not have property (2), although it has property (3).
\end{abstract}
\maketitle
\section{Introduction}
All Banach spaces we consider are over the field $\mathbb R$ of real numbers and infinite-dimensional, if not stated otherwise. For a Banach space $X$, let $B_X$, $S_X$, and $X^\ast$ denote the unit ball, the unit sphere, and the dual of $X$, respectively.
By a \emph{slice} of $B_X$ we mean a set of the form
\[
S(x^\ast,\alpha\}=\{x\in B_X\colon x^\ast(x)>1-\alpha\},
\]
where $x^\ast\in S_{X^\ast}$ and $\alpha>0$. A \emph{convex combination of slices} of $B_X$ is a set of the form
\[
\sum_{i=1}^n\lambda_i S_i,
\]
where $n\in\mathbb N$, $\lambda_1,\dotsc,\lambda_n\geq 0$ with $\sum_{i=1}^n \lambda_i=1$, and $S_1,\dotsc,S_n$ are slices of $B_X$.

By Bourgain's lemma \cite[Lemma II.1]{ggms}, every nonempty relatively weakly open subset of the unit ball of a Banach space contains a convex combination of slices of the unit ball.
But the converse is not always true: a convex combination of slices of the unit ball in general need not have a nonempty relatively weakly open subset (see, e.g., \cite[Remark IV.5 p.~48]{ggms}. Although in some Banach spaces every convex combination of slices of the unit ball has nonempty interior in relative weak topology of the unit ball, e.g., in $C(K)$, for scattered compact $K$, in particular, in $c$, and in $c_0$, as it is shown in a very recent preprint \cite{AL}. In fact, \cite[Theorems~2.3 and 2.4]{AL} asserts that in $C(K)$, for scattered compact $K$, and in $c_0$, every convex combination of slices of the unit ball is relatively weakly open.

Every nonempty relatively weakly open subset of the unit ball of a Banach space intersects the unit sphere. Therefore a convex combination of slices has to intersect the unit sphere in order to have nonempty relative weak interior.
One can immediately verify that in the spaces $C[0,1]$ and $L_1[0,1]$ every convex combination of slices of the unit ball intersects the unit sphere; the same is true for $\ell_\infty$ (see, e.g., \cite[Examples~3.2 and 3.3]{AL}.

Let us consider the following three properties for a Banach space:
\begin{itemize}
    \item[(P1)] \emph{every convex combination of slices of the unit ball is open in the relative weak topology (of the unit ball);}
    \item[(P2)] \emph{every convex combination of slices of the unit ball has nonempty interior in the relative weak topology (of the unit ball);}
    \item[(P3)] \emph{every convex combination of slices of the unit ball intersects the unit sphere.}
\end{itemize}

Clearly, one has
\[
(\rm P1)\implies(\rm P2)\implies(\rm P3).
\]
In this paper we prove that these three properties are all different from each other. In fact, their difference appears by studying how they act under direct sums with absolute norms.

We also show that the space $C[0,1]$ does not have property (P2), although it has property (P3), according to the remark above. This, of course, also implies the difference of (P2) and (P3).

Recall that a norm $N$ on $\mathbb R^2$ is called \emph{absolute} (see \cite{bonsall_numerical_1973}) if
\[
N(a,b)=N(|a|,|b|)\qquad\text{for all $(a,b)\in\mathbb R^2$,}
\]
and \emph{normalized} if
\[
N(1,0)=N(0,1)=1.
\]
For example, the $p$-norm $\|\cdot\|_p$ is absolute and normalized for every $p\in[1,\infty]$. If $N$ is an absolute normalized norm on $\mathbb R^2$ (see \cite[Lemmata~21.1 and 21.2]{bonsall_numerical_1973}), then 

\begin{itemize}
\item $\|(a,b)\|_\infty\leq N(a,b)\leq \|(a,b)\|_1$ for all $(a,b)\in\mathbb R^2$;
\item if $(a,b),(c,d)\in\mathbb R^2$ with $|a|\leq |c|$ and $|b|\leq |d|$, then \[N(a,b)\leq N(c,d);\]
\item the dual norm $N^\ast$ on $\mathbb \R^2$ defined by 
\[
N^\ast(c,d)=\max_{N(a,b)\leq 1}(|ac|+|bd|) \qquad\text{for all $(c,d)\in\mathbb R^2$}
\]
is also absolute and normalized. Note that $(N^\ast)^\ast=N$.
\end{itemize}

If $X$ and $Y$ are Banach spaces and $N$ is an absolute normalized norm on $\R^2$, then we denote by $X\oplus_N Y$ the direct sum $X\oplus Y$ equipped with the norm
\[
\|(x,y)\|_N=N(\|x\|,\|y\|) \qquad\text{for all $x\in X$ and $y\in Y$}.
\]
In the special case where $N$ is the $\ell_p$-norm, we write $X\oplus_p Y$.
Note that $(X\oplus_N Y)^\ast=X^\ast\oplus_{N^\ast} Y^\ast$.
\begin{definition}
We say that an absolute normalized norm on $N$ on $\mathbb R^2$ has the \emph{positive strong diameter 2 property} if whenever $n\in \mathbb{N}$, positive $f_1,\dots,f_n\in S_{(\R^2, N^\ast)^\ast}$, $\alpha_1,\dots,\alpha_n>0$, and $\lambda_1,\dots, \lambda_n\geq 0$ with $\sum_{i=1}^n \lambda_i=1$ there are positive $(a_i,b_i)\in S(f_i, \alpha_i)$ such that
\[
N\Big(\sum_{i=1}^n \lambda_i(a_i,b_i)\Big)=1. 
\]
\end{definition}

\section{Stability results of properties (P1), (P2), and (P3)\\ under direct sums}
We start with the following observation, which is a our joint result with Trond A.~Abrahamsen and Vegard Lima.

\begin{prop}\label{P11sum}
Let $X$ and $Y$ be Banach spaces, and let $N$ be an absolute normalized norm different from the $\infty$-norm, i.e., $N(1,1)>1$. Then $X\oplus_N Y$ does not have property (P1).
\end{prop}

\begin{proof}
Put $Z=X\oplus_N Y$, and let $x^\ast\in S_{X^\ast}$ and $y^\ast\in S_{Y^\ast}$. Fix $\varepsilon>0$ such that $N(1,1-\varepsilon)>1$. Consider the following four slices of $B_Z$:
\begin{align*}
S_1&=S((x^\ast,0),\varepsilon),\\
S_2&=S((-x^\ast,0),\varepsilon),\\
S_3&=S((0,y^\ast),\varepsilon),\\
S_4&=S((0,-y^\ast),\varepsilon).
\end{align*}
Define \[C=\frac14\big(S_1+S_2+S_3+S_4\big).\] We show that $C$ is not relatively weakly open in $B_Z$.

First, note that $(0,0)\in C$. To see this, choose $x\in B_X$ and $y\in B_Y$ with $x^\ast(x)>1-\varepsilon$ and $y^\ast(y)>1-\varepsilon$. Clearly, $(x,0)\in S_1$, $(-x,0)\in S_2$, $(0,y)\in S_3$, and $(0,-y)\in S_4$. Therefore 
\[(0,0)=\frac14\big((x,0)+(-x,0)+(0,y)+(0,-y)\big)\in C.\]

Suppose that $C$ is relatively weakly open. Then $(0,0)$ has an open basic neighbourhood $W\subset C$ in the relative weak topology of $B_X$, say
\[W=\{w\in B_Z\colon |z_i^\ast(w)|<\delta,\ i=1,\dotsc,n\},\]
where $n\in\mathbb N$, $z_i^\ast=(x_i^\ast,y_i^\ast)\in Z^\ast$, and $\delta>0$.
Define \[U=\{u\in B_X\colon |x_i^\ast(u)|<\delta,\ i=1,\dotsc,n\}.\] Then $U$ is a relatively weakly open neighborhood of $0\in X$; therefore $U$ intersects the unit sphere of $X$. Fix an $u\in S_X\cap U$. Then clearly $(u,0)\in W$. Since $W\subset C$, this implies that there exist $(x_j,y_j)\in S_j$, $j=1,\dotsc,4$, such that \[(u,0)=\frac14\big((x_1,y_1)+(x_2,y_2)+(x_3,y_3)+(x_4,y_4)\big).\] 
To obtain a contradiction, consider $(x_3,y_3)$. The conditions $\|u\|=1$ and $(x_3,y_3)\in S_3$ imply that $\|x_3\|=1$ and $\|y_3\|>1-\varepsilon$; hence
\[1= N(\|x_3\|,\|y_3\|)\geq N(1,1-\varepsilon)>1,\]
a contradiction.
\end{proof}

In contrast to property (P1), property (P3) is stable under $1$-sums.

\begin{prop}\label{P31sum}
Let $X$ and $Y$ be Banach spaces with property (P3), and let $N$ be an absolute normalized norm with the positive strong diameter 2 property. Then $X\oplus_N Y$ has property (P3).
\end{prop}

\begin{proof}
We follow an idea from \cite{HPP}. Put $Z=X\oplus_N Y$. Consider a convex combination of slices of $B_Z$, say $\sum_{i=1}^n\lambda_i S_i$, where $S_i=S((x_i^\ast,y_i^\ast),\alpha_i)$. Let
\[S_i^X=S(\frac{x_i^\ast}{\|x_i^\ast\|},\frac{\alpha_i}{2})\quad \text{if $x_i^\ast\not=0$,}\quad\qquad\text{and}\quad\qquad S_i^X=B_X\quad\text{if $x_i^\ast=0$,}\]
and
\[S_i^Y=S(\frac{y_i^\ast}{\|y_i^\ast\|},\frac{\alpha_i}{2})\quad\text{if $y_i^\ast\not=0$,}\quad\qquad \text{and}\quad\qquad S_i^Y=B_Y\quad\text{if $y_i^\ast=0$}.\]
Since $N$ has the positive strong diameter 2 property, there exist $a_i,b_i\geq0$ such that $N(a_i,b_i)=1$, $a_i\|x_i^\ast\|+b_i\|y_i^\ast\|\geq1-\alpha_i/2$, and $N(\sum_{i=1}^n\lambda_i(a_i,b_i))=1$. Define
\[\mu=\sum_{i=1}^n\lambda_ia_i\quad\text{and}\quad \nu=\sum_{i=1}^n\lambda_i b_i.\]

Assume first that $\mu\not=0$ and $\nu\not=0$.
Define 
\[\mu_i=\dfrac{\lambda_ia_i}{\mu}\quad\text{and}\quad\nu_i=\dfrac{\lambda_ib_i}{\nu}.\]
Since $X$ and $Y$ have property (P3), there exist
$x=\sum_{i=1}^n\mu_ix_i\in S_X$ and $y=\sum_{i=1}^n\nu_iy_i\in S_Y$, where $x_i\in S_i^X$ and $y_i\in S_i^Y$.
Now $(\mu x,\nu y)\in S_Z\cap \sum_{i=1}^n\lambda_i S_i$, because
\[
\|(\mu x,\nu y)\|_N=N(\mu,\nu)=1,
\]
$\mu x=\sum_{i=1}^n\lambda_i a_i x_i$, $\nu y=\sum_{i=1}^n \lambda_i b_i y_i$, and $(a_i x_i,b_i y_i)\in S_i$.

Consider now the case, where $\mu=0$. The case, where $\nu=0$ is similar and is therefore omitted. We can assume that $\lambda_i>0$ for every $i$. Since $\mu=0$, we have $a_i=0$ and $b_i=1$ for every $i$. Since $Y$ has property (P3), there exists $y\in S_Y\cap\sum_{i=1}^n \lambda_i S_i^Y$. Then clearly $(0,y)\in S_Z\cap\sum_{i=1}^n \lambda_i S_i$.
\end{proof}

The following theorem together with the above ones will lead to our main conclusion in this section.

\begin{theorem}\label{PP1sum}
Let $X$ and $Y$ be Banach spaces with the property that norm-one elements in convex combinations of open slices of the closed unit ball are interior points
of these combinations in the relative weak topology (of the closed unit ball).
Then also the $1$-sum $X\oplus_1 Y$ has the same property.
\end{theorem}
\begin{proof}[Proof of Theorem]
Put $Z:=X\oplus_1 Y$, and let $n\in\N$, let
\[
S_1:=S(\zs_1,\gamma_1),\,\dotsc,\, S_n:=S(\zs_n,\gamma_n)\quad\text{ with $\zs_i=(\xs_i,\ys_i)\in S_{\Zs}$}
\]
be open slices of $B_Z$,
let $\lambda_1,\dotsc,\lambda_n\geq 0$ satisfy $\sum\limits_{i=1}^n{\lambda_i}=\nobreak1$,
and let $z=(x,y)=\sum\limits_{i=1}^n\lambda_i z_i$, where $z_i=(x_i,y_i)\in S_i$, satisfy $\|z\|=\|x\|+\|y\|=1$.
Notice that $x=\sum\limits_{i=1}^n\lambda_i x_i$ and  $y=\sum\limits_{i=1}^n\lambda_i y_i$, and
\[
\|x\|=\sum_{i=1}^n\lambda_i\|x_i\|
\quad\text{and}\quad
\|y\|=\sum_{i=1}^n\lambda_i\|y_i\|.
\]
We must find a finite subset $\mathcal{C}$ of $\Zs$ and a $\delta>0$ such that, for
\begin{equation}\label{eq: definition of W}
W:=\bigl\{w\in B_Z\colon\,|\ws(w-z)|<\delta\quad\text{for every $\ws\in\mathcal{C}$}\bigr\},
\end{equation}
one has $W\subset\sum\limits_{i=1}^n\lambda_i S_i$.

For every $i\in\{1,\dotsc,n\}$, putting $\delta_i=\zs_i(z_i)-(1-\gamma_i)$,
\[
\widehat{x}_i=
\begin{cases}
\dfrac{x_i}{\|x_i\|},&\quad\text{if $x_i\not=0$,}\\
0,&\quad\text{if $x_i=0$,}
\end{cases}
\qquad\qquad
\widehat{y}_i=
\begin{cases}
\dfrac{y_i}{\|y_i\|},&\quad\text{if $y_i\not=0$,}\\
0,&\quad\text{if $y_i=0$,}
\end{cases}
\]
\[
S^X_i:=\bigl\{x\in B_X\colon\,\xs_i(x)>\xs_i(\widehat{x}_i)-\tfrac{\delta_i}2\bigr\},
\quad
S^Y_i:=\bigl\{y\in B_Y\colon\,\ys_i(y)>\ys_i(\widehat{y}_i)-\tfrac{\delta_i}2\bigr\},
\]
observe that $S^X_i$ and $S^Y_i$ are slices of, respectively, $B_X$ and $B_Y$ with $\widehat{x}_i\in S^X_i$ and $\widehat{y}_i\in S^Y_i$,
and
\[
\bigl(\|x_i\|\,S^X_i\bigr)\times\bigl(\|y_i\|\,S^Y_i\bigr)\subset\bigl\{w\in B_Z\colon\,\zs_i(w)>1-\gamma_i+\tfrac{\delta_i}{2}\bigr\}\subset S_i.
\]
Indeed, whenever $u\in S^X_i$ and $v\in S^Y_i$, one has
\begin{align*}
\zs_i\bigl(\|x_i\|u,\|y_i\|v\bigr)
&=\|x_i\|\xs_i(u)+\|y_i\|\ys_i(v)\\
&>\|x_i\|\Bigl(\xs_i(\widehat{x}_i)-\frac{\delta_i}2\Bigr)+\|y_i\|\Bigl(\ys_i(\widehat{y}_i)-\frac{\delta_i}2\Bigr)\\
&=\xs_i(x_i)+\ys_i(y_i)-\bigl(\|x_i\|+\|y_i\|\bigr)\frac{\delta_i}2
=\zs_i(z_i)-\frac{\delta_i}2\\
&=1-\gamma_i+\frac{\delta_i}{2}.
\end{align*}

We only consider the case when both $\|x\|\not=0$ and $\|y\|\not=0$. (The case when $\|x\|=0$ or $\|y\|=0$ can be handled similarly and is, in fact, simpler.)

For every $i\in\{1,\dotsc,n\}$, letting
\[
\alpha_i:=\frac{\lambda_i\|x_i\|}{\|x\|}
\quad\text{and}\quad
\beta_i:=\frac{\lambda_i\|y_i\|}{\|y\|},
\]
one has $\alpha_i,\beta_i\geq0$, and $\sum\limits_{i=1}^n\alpha_i=\sum\limits_{i=1}^n\beta_i=1$.
Observing that
\begin{align*}
\frac{x}{\|x\|}=\sum_{i=1}^n\frac{\lambda_i\|x_i\|}{\|x\|}\,\widehat{x}_i\in\sum_{i=1}^n\alpha_i S^X_i
\quad\text{and}\quad
\frac{y}{\|y\|}=\sum_{i=1}^n\frac{\lambda_i\|y_i\|}{\|y\|}\,\widehat{y}_i\in\sum_{i=1}^n\beta_i S^Y_i,
\end{align*}
there are finite subsets $\mathcal{A}$ of $S_{\Xs}$ and $\mathcal{B}$ of $S_{\Ys}$, and a $\gamma\in(0,1)$ such that, for
\[
U:=\biggl\{u\in B_X\colon\,\Bigl|\us\Bigl(u-\frac{x}{\|x\|}\Bigr)\Bigr|<\gamma\quad\text{for every $\us\in\mathcal{A}$}\biggr\}
\]
and
\[
V:=\biggl\{v\in B_Y\colon\,\Bigl|\vs\Bigl(v-\frac{y}{\|y\|}\Bigr)\Bigr|<\gamma\quad\text{for every $\vs\in\mathcal{B}$}\biggr\},
\]
one has $U\subset\sum\limits_{i=1}^n\alpha_i S^X_i$ and $V\subset\sum\limits_{i=1}^n\beta_i S^Y_i$.

Choose $\xs\in S_{\Xs}$ and $\ys\in S_{\Ys}$ so that $\xs(x)=\|x\|$ and $\ys(y)=\|y\|$.
Let $W$ be defined by \eqref{eq: definition of W} where $\delta>0$ satisfies
\[
\delta<\frac{\gamma\min\{\|x\|,\|y\|\}}{2},\quad 
\qquad\text{and}\qquad
2\delta<\frac{\gamma_i}{2}\quad\text{for all $i\in\{1,\dotsc,n\}$.}
\]
and
\[
\mathcal{C}:=\bigl\{(\us,0),(0,\vs)\colon\,\us\in\mathcal{A}\cup\{\xs\},\,\vs\in\mathcal{B}\cup\{\ys\}\bigr\}.
\]
Now suppose that $w=(u,v)\in W$. Then
\[
\|u\|\geq\xs(u)=(\xs,0)(w)>(\xs,0)(z)-\delta=\xs(x)-\delta=\|x\|-\delta,
\]
and, similarly, $\|v\|>\|y\|-\delta$; thus
\[
1\geq\|u\|+\|v\|>\|u\|+\|y\|-\delta,
\]
whence $\|u\|<1-\|y\|+\delta=\|x\|+\delta$, and, similarly, $\|v\|<\|y\|+\delta$.

There are two alternatives:
\begin{itemize}
\item[(1)]
$\|u\|\leq\|x\|$ and $\|v\|\leq\|y\|$;
\item[(2)]
$\|u\|>\|x\|$ or $\|v\|>\|y\|$.
\end{itemize}

Let us consider the case (2). In this case, one may assume that $\|u\|=\|x\|+\eps$  and $\|v\|=\|y\|-\eps_1$ where $0<\eps\leq\eps_1<\delta$. Setting
\[
\widehat{u}:=\dfrac{u}{\|u\|},\quad
u_0:=\|x\|\widehat{u},\quad
\widehat{v}:=\dfrac{v}{\|v\|},\quad
v_0:=\|y\|\widehat{v},
\]
one has
\[
u=\|u\|\widehat{u}=(\|x\|+\eps)\widehat{u}=u_0+\eps\widehat{u}
\]
and
\[
v=\|v\|\widehat{v}=(\|y\|-\eps_1)\widehat{v}=v_0-\eps_1\widehat{v}=\Bigl(1-\frac{\eps_1}{\|y\|}\Bigr)v_0.
\]
Since $\dfrac{u_0}{\|x\|}\in S_X$ and, for every $\us\in\mathcal{A}$,
\begin{align*}
\Bigl|\us\Bigl(\dfrac{u_0}{\|x\|}-\frac{x}{\|x\|}\Bigr)\Bigr|
&\leq\frac1{\|x\|}\bigl(|\us(u-x)|+\eps|\us(\widehat{u})|\bigr)
<\frac{2\delta}{\|x\|}<\gamma,
\end{align*}
one has $\dfrac{u_0}{\|x\|}\in U$, thus there are $u_i\in S^X_i$, $i=1,\dotsc,n$, such that
\begin{equation}\label{eq: u_0=||x||sum alpha_i u_i}
u_0=\|x\|\sum_{i=1}^n\alpha_i u_i=\sum_{i=1}^n\lambda_i\|x_i\|u_i.
\end{equation}
Similarly, since $\dfrac{v_0}{\|y\|}\in S_Y$ and, for every $\vs\in\mathcal{B}$,
\begin{align*}
\Bigl|\vs\Bigl(\dfrac{v_0}{\|y\|}-\frac{y}{\|y\|}\Bigr)\Bigr|
&\leq\frac1{\|y\|}\bigl(|\vs(v-y)|+\eps_1|\vs(\widehat{v})|\bigr)
<\frac{2\delta}{\|y\|}<\gamma,
\end{align*}
one has $\dfrac{v_0}{\|y\|}\in V$, thus there are $v_i\in S^Y_i$, $i=1,\dotsc,n$, such that
\begin{equation}\label{eq: v_0=||y||sum beta_i v_i}
v_0=\|y\|\sum_{i=1}^n\beta_i v_i=\sum_{i=1}^n\lambda_i\|y_i\|v_i.
\end{equation}
Now,
\[
v=\Bigl(1-\frac{\eps_1}{\|y\|}\Bigr)v_0=\sum_{i=1}^n\lambda_i\Bigl(\|y_i\|-\frac{\eps_1\|y_i\|}{\|y\|}\Bigr)v_i
\]
and
\[
u=u_0+\eps\widehat{u}=\sum_{i=1}^n\lambda_i\|x_i\|u_i+\biggl(\sum_{i=1}^n\frac{\lambda_i\|y_i\|}{\|y\|}\biggr)\eps\widehat{u}
=\sum_{i=1}^n\lambda_i\biggl(\|x_i\|u_i+\frac{\eps\|y_i\|}{\|y\|}\widehat{u}\biggr).
\]
Thus, setting
\[
\widetilde{u}_i:=\dfrac{\eps\|y_i\|}{\|y\|}\widehat{u},\quad
\widetilde{v}_i:=\dfrac{\eps_1\|y_i\|}{\|y\|}v_i,\quad\text{and}\quad
w_i=\bigl(\|x_i\|u_i+\widetilde{u}_i,\|y_i\|v_i-\widetilde{v}_i\bigr),
\]
one has
\begin{align*}
(u,v)&=\sum_{i=1}^n\lambda_i\biggl(\|x_i\|u_i+\frac{\eps\|y_i\|}{\|y\|}\widehat{u},\Bigl(\|y_i\|-\frac{\eps_1\|y_i\|}{\|y\|}\Bigr)v_i\biggr)\\
&=\sum_{i=1}^n\lambda_i\bigl(\|x_i\|u_i+\widetilde{u}_i,\|y_i\|v_i-\widetilde{v}_i\bigr)
=\sum\limits_{i=1}^n\lambda_i w_i,
\end{align*}
and it remains to observe that $w_i\in S_i$ for every $i\in\{1,\dotsc,n\}$.
Indeed,
\begin{align*}
\|w_i\|&=\bigl\|\|x_i\|u_i+\widetilde{u}_i\bigr\|+\|y_i\|-\frac{\eps_1\|y_i\|}{\|y\|}\\
&\leq\|x_i\|+\|\widetilde{u}_i\|+\|y_i\|-\frac{\eps_1\|y_i\|}{\|y\|}
=1+\dfrac{\eps\|y_i\|}{\|y\|}-\frac{\eps_1\|y_i\|}{\|y\|}\leq1
\end{align*}
and
\begin{align*}
\zs_i(w_i)
&=\zs_i\bigl(\|x_i\|u_i,\|y_i\|v_i\bigr)+\zs_i(\widetilde{u}_i,-\widetilde{v}_i)\\
&>\zs_i(z_i)-\frac{\delta_i}2-\|\widetilde{u}_i\|-\|\widetilde{v}_i\|
>1-\gamma_i+\frac{\delta_i}2-2\delta
>1-\gamma_i.
\end{align*}

In the case (1), one immediately verifies that $\dfrac{u}{\|x\|}\in U$ and $\dfrac{v}{\|y\|}\in V$, thus one has the representations \eqref{eq: u_0=||x||sum alpha_i u_i}
and \eqref{eq: v_0=||y||sum beta_i v_i} with $u_i\in S^X_i$ and $v_i\in S^Y_i$ for $u_0=u$ and $v_0=v$, and it follows that
\[
(u,v)=\sum_{i=1}^n\lambda_i\bigl(\|x_i\|u_i,\|y_i\|v_i\bigr)
\in\sum_{i=1}^n\lambda_i\Bigl(\bigl(\|x_i\|\,S^X_i\bigr)\times\bigl(\|y_i\|\,S^Y_i\bigr)\Bigr)
\subset\sum_{i=1}^n\lambda_i S_i.
\]
\end{proof}

The next result follows easily from Proposition~\ref{P31sum} and Theorem~\ref{PP1sum}.

\begin{prop}\label{P21sum}
Let $X$ and $Y$ be Banach spaces with property (P1). Then $X\oplus_1 Y$ has property (P2).
\end{prop}

\begin{proof}
Let $C$ be a convex combination of slices of the unit ball of $X\oplus_1 Y$. By Proposition~\ref{P31sum}, there exists a norm-one $z\in C$. By Theorem~\ref{PP1sum}, the point $z$ in an interior point of $C$ in the relative weak topology of the unit ball.
\end{proof}

Propositions~\ref{P11sum} and \ref{P21sum} immediately imply our main result in this section.

\begin{cor}
Properties (P1) and (P2) are different.
\end{cor}

We now turn to $\infty$-sums.

Property (P1) is stable under $\infty$-sums.

\begin{prop}\label{P1infty}
Let $X$ and $Y$ be Banach spaces. Then $X$ and $Y$ have property  (P1) if and only if $X\oplus_\infty Y$ has property (P1).
\end{prop}

\begin{proof}
Set $Z:=X\oplus_\infty Y$. Assume that $X$ and $Y$ have property (P1). Let $n\in\mathbb N$ and consider the slices 
\[
S_1=S(z_1^\ast,\alpha_1),\dotsc, S_n=S(z_n^\ast,\alpha_n),
\]
where $z_i^\ast=(x_i^\ast,y_i^\ast)\in S_{Z^\ast}$.
Let $\lambda_1,\dotsc,\lambda_n\geq0$ satisfy $\sum_{i=1}^n\lambda_i=1$.
We show that $\sum_{i=1}^n \lambda_i S_i$ is relatively weakly open. Fix $z=\sum_{i=1}^n \lambda_i z_i\in\sum_{i=1}^n \lambda_i S_i$, where $z_i=(x_i,y_i)\in S_i$.
Choose an $\varepsilon>0$ such that
\[z_i^\ast(z_i)-2\varepsilon>1-\alpha_i\qquad\text{for all $i\in\{1,\dotsc,n\}$}.\]
For every $i\in\{1,\dotsc,n\}$, define
\[S_i^X=\{u\in B_X\colon x_i^\ast(u)>x_i^\ast(x_i)-\varepsilon\}\]
and
\[S_i^Y=\{v\in B_Y\colon y_i^\ast(v)>y_i^\ast(y_i)-\varepsilon\}.\]
Observe that $S_i^X$ and $S_i^Y$ are slices of, respectively, $B_X$ and $B_Y$, and $x_i\in S_i^X$, $y_i\in S_i^Y$, and $S_i^X\times S_i^Y\subset S_i$. It follows that
\begin{align*}
z=\sum_{i=1}^n\lambda_i z_i&=\big(\sum_{i=1}^n\lambda_i x_i,\sum_{i=1}^n\lambda_i y_i\big)\\&\in \Big(\sum_{i=1}^n \lambda_i S_i^X\Big)\times \Big(\sum_{i=1}^n\lambda_i S_i^Y\Big)\\&=\sum_{i=1}^n\lambda_i \big(S_i^X\times S_i^Y\big)\\&\subset \sum_{i=1}^n \lambda_i S_i.
\end{align*}
Since $\Big(\sum_{i=1}^n \lambda_i S_i^X\Big)\times \Big(\sum_{i=1}^n\lambda_i S_i^Y\Big)$ is relatively weakly open in $B_Z$, we are done.

Assume now that $Z$ has property (P1). We only show that $X$ has property (P1), similarly one obtains that $Y$ has property (P1). Let $C$ be a convex combination of slices of $B_X$. Then $D:=C\times B_Y$ is a convex combination of slices of $B_Z$; therefore $D$ is relatively weakly open. Thus $C$ is relatively weakly open in $B_X$.
\end{proof}

The behavior of property (P2) under $\infty$-sums is the same as that of property (P1).
\begin{prop}\label{P2infty}
Let $X$ and $Y$ be Banach spaces. Then $X$ and $Y$ have property (P2) if and only if $X\oplus_\infty Y$ has property (P2).
\end{prop}

\begin{proof}
Set $Z=X\oplus_\infty Y$. Assume that $X$ and $Y$ have property (P2). Consider a convex combination of slices of $B_Z$, say $\sum_{i=1}^n\lambda_i S_i$, where $S_i=S((x_i^\ast,y_i^\ast),\alpha_i)$. Let 
\[
S_i^X=S(\frac{x_i^\ast}{\|x_i^\ast\|},\alpha_i)\quad\text{if $x_i^\ast\not=0$,}\qquad\quad\text{and}\quad\qquad S_i^X=B_X\quad\text{if $x_i^\ast=0$},
\]
and 
\[
S_i^Y=S(\frac{y_i^\ast}{\|y_i^\ast\|},\alpha_i)\quad\text{if $y_i^\ast\not=0$,}\qquad\quad\text{and}\quad\qquad S_i^Y=B_Y\quad\text{if $y_i^\ast=0$}.
\]
We have $S_i^X\times S_i^Y\subset S_i$, because for $u\in S_i^X$ and $v\in S_i^Y$,
\begin{align*}
x_i^\ast(u)+y_i^\ast(v)>\|x_i^\ast\|(1-\alpha_i)+\|y_i^\ast\|(1-\alpha_i)=1-\alpha_i.
\end{align*}
Since $\sum_{i=1}^n\lambda_iS_i^X$ and $\sum_{i=1}^n\lambda_i S_i^Y$ have nonempty relatively weak interior, there are $x\in B_X$, $y\in B_Y$, and relatively weakly open subsets $U\subset B_X$ and $V\subset B_Y$ such that 
\[x\in U\subset \sum_{i=1}^n\lambda_i S_i^X,\quad y\in V\subset \sum_{i=1}^n\lambda_i S_i^Y.\]
Clearly $U\times V\subset B_Z$ is relatively weakly open and 
\[(x,y)\in U\times V\subset \sum_{i=1}^n\lambda_i S_i^X\times \sum_{i=1}^n\lambda_i S_i^Y=\sum_{i=1}^n\lambda_i (S_i^X\times S_i^Y)\subset \sum_{i=1}^n\lambda_i S_i.\]

Assume now that $Z$ has property (P2). We only show that $X$ has property (P2), similarly one obtains that $Y$ has property (P2). Let $C$ be a convex combination of slices of $B_X$. Then $D:=C\times B_Y$ is a convex combination of slices of $B_Z$; therefore $D$ has nonempty interior in the relative weak topology of $B_Z$. Thus $C$ has nonempty interior in the relative weak topology of $B_X$.

\end{proof}

By Proposition~\ref{P1infty}, in order for the $\infty$-sum of two Banach spaces to have property (P1), it is necessary that both these spaces have (P1). By Proposition~\ref{P2infty}, the same is true for property (P2). For property (P3) in the same situation, it suffices that only one of these spaces  has (P3).

\begin{prop}\label{P3infty} Let $X$ and $Y$ be Banach spaces. If $X$ has property (P3), then $X\oplus_\infty Y$ has property (P3).
\end{prop}

\begin{proof}
Set $Z=X\oplus_\infty Y$. Consider a convex combination of slices of $B_Z$, say $\sum_{i=1}^n\lambda_i S_i$, where $S_i=S((x_i^\ast,y_i^\ast),\alpha_i)$. Let $(x,y)=\sum_{i=1}^n\lambda_i(x_i,y_i)\in \sum_{i=1}^n\lambda_i S_i$, where $(x_i,y_i)\in S_i$. Define \[S_i^X=\{u\in B_X\colon x_i^\ast(u)>1-\alpha_i-y_i^\ast(y_i)\}.\] Then $S_i^X$ is a slice of $B_X$ and $x_i\in S_i^X$. Since $X$ has property (P3), there exists         
$u=\sum_{i=1}^n\lambda_i u_i\in S_X$, where $u_i\in S_i^X$. It remains to obsreve that $(u_i,y_i)\in S_i$ and $\|\sum_{i=1}^n\lambda_i(u_i,y_i)\|_\infty=\|\sum_{i=1}^n\lambda_i u_i\|=1$.
\end{proof}

Since there exists Banach spaces without property (P3), Propositions~\ref{P2infty} and \ref{P3infty} immediately imply our second main result in this section.

\begin{cor}
Properties (P2) and (P3) are different.
\end{cor}

\section{$C_0(K)$ with $K$ non-scattered fails property (P2)}

\begin{thm}\label{thm: C_0(K) with K non-scattered}
Let a locally compact Hausdorff topological space $K$ admit an atomless regular finite Borel measure $\mu$.
Then $C_0(K)$ admits a convex combination of slices (of the closed unit ball) with empty interior in the relative weak topology (of the closed unit ball).
\end{thm}

\begin{rem}\label{rem: scatterednes}
A topological space $K$ is said to be \emph{scattered} (or \emph{disperse}) if every non-empty subset of $K$ has an isolated point (with respect to this subset).
(Equivalently, $K$ is scattered if it does not contain non-empty perfect subsets.) 
In \cite{Pel&Sem}, it was proven that, for a compact Hausdorff topological space $K$, the following assertions are equivalent:
\begin{itemize}
\item[{\rm(i)}]
\emph{$K$ is scattered;}
\item[{\rm(ii)}]
\emph{there exist no non-zero atomless regular Borel measures on $K$.}
\end{itemize}
In fact, the implication (i)$\Rightarrow$(ii) is due to Rudin \cite{Rudin}.
A more direct proof of the equivalence (i)$\Leftrightarrow$(ii) can be found in \cite{Knowles}
with the proof of (i)$\Rightarrow$(ii) being stunningly simple.

The equivalence (i)$\Leftrightarrow$(ii) easily generalizes to locally compact Hausdorff spaces
(see \cite{Luther} for details).
\end{rem}

\begin{rem}
In \cite{AL}, it has been proven that, \emph{for an infinite compact scattered Hausdorff topological space $K$, the Banach space $C(K)$ has property (P1).}
By Remark \ref{rem: scatterednes}, Theorem \ref{thm: C_0(K) with K non-scattered} complements this result:
\emph{$C(K)$ has property (P1) if and only if $K$ is scattered,} and \emph{$C(K)$ fails property (P2) if and only if $K$ is not scattered.}
\end{rem}

\begin{rem}\label{rem: conv. comb. of slices in C_0(K) reach the sphere}
Since every convex combination of slices of the unit ball of $C_0(K)$ (with $K$ infinite) can easily be seen to intersect the unit sphere,
the space $C_0(K)$ has property (P1). Thus, for a non-scattered $K$, the space $C_0(K)$ serves as an example of a Banach space with property (P3) but without (P2).
\end{rem}

The proof of Theorem \ref{thm: C_0(K) with K non-scattered} uses the following measure-theoretical lemma.

\begin{lem}\label{lem: n continuous measures}
Let $\Sigma$ be a $\sigma$-algebra of subsets of a non-empty set $\Omega$, let $n\in\N$, and let $\mu_1,\dotsc,\mu_n$ be atomless finite (non-negative) measures on $\Sigma$.
Then, whenever $E\in\Sigma$ and $\eps>0$, there are disjoint $A,B\in\Sigma$ such that $A\cup B=E$ and $|\mu_i(A)-\mu_i(B)|\leq\eps$ for every $i\in\{1,\dotsc,n\}$.
\end{lem}

\begin{rem}
For $n=1$, Lemma \ref{lem: n continuous measures} is well known to hold even with $\eps=0$ (see, e.g., \cite[Theorem~10.52]{Aliprantis&}). We do not know, whether it remains true with $\eps=0$ for $n\geq2$.
\end{rem}

\begin{proof}[Proof of Lemma \ref{lem: n continuous measures}]
We use induction on $n$. First, for $n=1$, the lemma is well known to hold even with $\eps=0$  (see, e.g., \cite[Theorem~10.52]{Aliprantis&}).
Now assume that the lemma holds for $n=m$ for some $m\in\N$. To see that it also holds for $n=m+1$,
let  $\mu_1,\dotsc,\mu_m,\mu$ be atomless finite (non-negative) measures on $\Sigma$, let $E\in\Sigma$, and let $\eps>0$.
Since $\mu$ is atomless, we can, for some $k\in\N$, represent $E$ as a union $E=\bigcup\limits_{j=1}^k E_j$ of pairwise disjoint sets $E_j,\dotsc,E_k\in\Sigma$ satisfying
$\mu(E_1),\dotsc,\mu(E_k)<\eps$. By our assumption, we can represent each $E_j$ as a union $E_j=C_j\cup D_j$ of disjoint sets $C_j,D_j\in\Sigma$ such that
$|\mu_i(C_j)-\mu_i(D_j)|\leq\dfrac{\eps}{2^j}$ for every $i\in\{1,\dotsc,m\}$. We may assume that $\mu(C_j)\geq\mu(D_j)$ for every $j\in\{1,\dotsc,k\}$.
Define $A_1=C_1$ and $B_1=D_1$, and proceed as follows: provided $l\in\{1,\dotsc,k-1\}$ and $A_1,\dotsc, A_l,B_1,\dotsc, B_l$, define
\begin{align*}
A_{l+1}&:=C_{l+1}&&\quad\text{and}\quad &B_{l+1}&:=D_{l+1}&&\qquad\text{if}\quad &\sum\limits_{j=1}^l\mu(A_j)&\leq\sum\limits_{j=1}^l\mu(B_j);\\
A_{l+1}&:=D_{l+1}&&\quad\text{and}\quad &B_{l+1}&:=C_{l+1}&&\qquad\text{if}\quad &\sum\limits_{j=1}^l\mu(A_j)&>\sum\limits_{j=1}^l\mu(B_j).
\end{align*}
Now set $A:=\bigcup\limits_{j=1}^k A_j$ and $B:=\bigcup\limits_{j=1}^k B_j$. Then $E=A\cup B$ and, for every $i\in\{1,\dotsc,m\}$,
\begin{align*}
|\mu_i(A)-\mu_i(B)|
\leq\sum_{j=1}^k|\mu_i(A_j)-\mu_i(B_j)|\leq\sum_{j=1}^k\dfrac{\eps}{2^j}<\eps.
\end{align*}
Also, clearly $|\mu(A)-\mu(B)|<\eps$, and the proof is complete.
\end{proof}

\begin{proof}[Proof of Theorem \ref{thm: C_0(K) with K non-scattered}]
Throughout the proof, functionals in $C_0(K)^\ast$ will be identified with their representing (regular finite signed) Borel measures.
The Borel $\sigma$-algebra of $K$ will be denoted by $\mathcal{B}_K$.

Without loss of generality, one may assume that $\mu(K)=1$.
Since $\mu$ is atomless, there are pairwise disjoint $A, B, C\in \mathcal{B}_K$ whose union is $K$ and $\mu(A)=\mu(B)=\mu(C)=\dfrac{1}{3}$.
Define (signed) Borel measures $\mu_1$ and $\mu_2$ on~$K$ by
\[
\mu_i(E)=\mu(E\cap A)+(-1)^{i-1}\mu(E\cap B)-\mu(E\cap C), \quad E\in \mathcal{B}_K,\qquad i=1,2,
\]
and let $0<\eps<\dfrac1{18}$. Consider the slices $S_1:=S(\mu_1,\eps^2)$ and $S_2:=S(\mu_2,\eps^2)$ of the closed unit ball of $C_0(K)$.
For any $u\in B_{C_0(K)}$, set
\begin{align*}
B^u_1:=B\cap\{u\leq1-\eps\},
\quad
B^u_{-1}:=B\cap\{u\geq-1+\eps\},
\quad
B^u_0:=B\cap\{|u|\geq\eps\}.
\end{align*}
Observe that, whenever $x\in S_1$, one has $\mu(B^x_{1})<\eps$, because otherwise one would have
\begin{align*}
\mu_1(x)
&=\int_A x\,d\mu+\int_B x\,d\mu-\int_C x\,d\mu\\
&\leq\mu(A)+\mu(C)+\mu(B\setminus B^x_{1})+(1-\eps)\mu(B^x_{1})\\
&=1-\mu(B^x_{1})+(1-\eps)\mu(B^x_{1})=1-\eps\mu(B^x_{1})\\
&\leq1-\eps^2,
\end{align*}
a contradiction. Similarly, whenever $y\in S_2$, one has $\mu(B^y_{-1})<\eps$, because otherwise one would have
\begin{align*}
\mu_2(y)
&=\int_A y\,d\mu-\int_B y\,d\mu-\int_C y\,d\mu\\
&\leq\mu(A)+\mu(C)+\mu(B\setminus B^y_{-1})+(1-\eps)\mu(B^y_{-1})\\
&=1-\mu(B^y_{-1})+(1-\eps)\mu(B^y_{-1})=1-\eps\mu(B^y_{-1})\\
&\leq1-\eps^2.
\end{align*}
It follows that $\mu(B^u_0)<2\eps$ for every $u\in \dfrac12S_1+\dfrac12S_2$.
Indeed, let $u=\dfrac12x+\dfrac12y$ where $x\in S_1$ and $y\in S_2$.
Then $B\setminus B^u_0\supset (B\setminus B^x_1)\cap(B\setminus B^y_{-1})=:\widehat{B}$, because, whenever $t\in\widehat{B}$, one has
$x(t)=1-\eps_1$ and $y(t)=-1+\eps_2$ for some $\eps_1,\eps_2\in[0,\eps)$, and thus
\begin{align*}
u(t)=\frac12(1-\eps_1)+\frac12(-1+\eps_2)=\frac{\eps_2-\eps_1}2.
\end{align*}
Hence $B^u_0\subset B\setminus \widehat{B}=B^x_1\cup B^y_{-1}$ and $\mu(B^u_0)\leq\mu(B^x_1)+\mu(B^y_{-1})<2\eps$.

Now let $z\in \dfrac12S_1+\dfrac12S_2$ be arbitrary, let $\mathcal{C}$ be a finite subset of the dual unit sphere $S_{C_0(K)^\ast}$, and let $\gamma>0$.
It suffices to show that the set
\[
U:=\{u\in B_{C_0(K)}\colon\,\text{$|\nu(u-z)|<\gamma$ for every $\nu\in\mathcal{C}$}\}
\]
contains an element which is not in $\dfrac12S_1+\dfrac12S_2$.
Without loss of generality, one may assume that, for every functional in $\mathcal{C}$, its representing measure is either atomless or a Dirac measure.
So one may assume that, for some $n\in\N$ and some finite subset $F\subset K$,
one has $\mathcal{C}=\bigl\{\nu_j,\,\delta_t\colon\,j\in\{1,\dotsc,n\},\,t\in F\bigr\}$, where each $\nu_j$ is atomless
and $\delta_{t}$ is the Dirac measure at $t\in F$.

Define $E:=A\cup C\cup B^z_0$, $D:=B\setminus B^z_0=B\setminus E=\Omega\setminus E$, $\nu:=\mu+|\nu_1|+\dotsb+|\nu_n|$,
and let $0<\delta<\min\Bigl\{\eps,\dfrac{\gamma}{9}\Bigr\}$.
Using Hahn's decomposition, one can represent $D$ as a finite union $D=\bigcup\limits_{k=1}^{m}D_k$
where $D_1,\dotsc,D_m$ are disjoint Borel sets such that, for every $D_k$, each  $\nu_j$ preserves its sign on Borel subsets of $D_k$.
By Lemma \ref{lem: n continuous measures}, each $D_k$ can be decomposed as a disjoint union $D_k=D^1_k\cup D^2_k$ so that, for every $j\in\{1,\dots,n\}$,
one has $|\nu_j(D^1_k)-\nu_j(D^2_k)|<\dfrac{\delta}{m}$, and thus, setting $D^i:=\bigcup\limits_{k=1}^m D^i_k$, $i=1,2$,
\[
|\nu_j(D^1)-\nu_j(D^2)|\leq\sum_{k=1}^m|\nu_j(D^1_k)-\nu_j(D^2_k)|<\delta.
\]
By the regularity of $\nu$, there are compact subsets $K_E\subset E$ and $K^i_k\subset D^i_k$
such that $\nu(E\setminus K_E)<\delta$ and  $\nu(D^i_k\setminus K^i_k)<\dfrac{\delta}{2m}$ for every $i\in\{1,2\}$ and every $k\in\{1,\dotsc,m\}$,
and thus, setting $K^i:=\bigcup\limits_{k=1}^m K^i_k$, $i=1,2$, and $\widehat{K}:=K^1\cup K^2$,
\[
\nu(D\setminus\widehat{K})=\sum_{i=1}^2\sum_{k=1}^m\nu(D^i_k\setminus K^i_k)<\delta.
\]
Again, by the regularity of $\nu$, there is an open set $U\supset F\cap D$ such that $\nu(U)<\delta$.
By Tietze's extension theorem, there is a $u\in B_{C_0(K)}$ such that
\[
u(t)=
\begin{cases}
z(t)&\quad\text{if $t\in K_E\cup F$;}\\
z(t)+1-\eps&\quad\text{if $t\in K^1\setminus U$;}\\
z(t)-1+\eps&\quad\text{if $t\in K^2\setminus U$.}
\end{cases}
\]

One has $u\in U$, because $|\delta_t(u-z)|=|u(t)-z(t)|=0$ for every $t\in F$ and $|\nu_j(u-z)|<\gamma$ for every $j\in\{1,\dotsc,n\}$.
Indeed, since
\begin{align*}
\int_{\widehat{K}\setminus U}(u-z)\,d\nu_j
&=\int_{K^1\setminus U}(u-z)\,d\nu_j+\int_{K^2\setminus U}(u-z)\,d\nu_j\\
&=(1-\eps)\nu_j(K^1\setminus U)-(1-\eps)\nu_j(K^2\setminus U)\\
&=(1-\eps)\bigl(\nu_j(K^1\setminus U)-\nu_j(K^2\setminus U)\bigr)
\end{align*}
and
\begin{align*}
|\nu_j&(K^1\setminus U)-\nu_j(K^2\setminus U)|\\
&\leq|\nu_j(K^1)-\nu_j(K^2)|+|\nu_j(K^1\cap U)|+|\nu_j(K^2\cap U)|\\
&\leq|\nu_j(D^1)-\nu_j(D^2)|+|\nu_j(D^1\setminus K^1)|+|\nu_j(D^2\setminus K^2)|+|\nu_j|(U)\\
&\leq\delta+|\nu_j|(D\setminus\widehat{K})|+\delta\\
&<3\delta,
\end{align*}
it follows that
\begin{align*}
\biggl|\int_{\widehat{K}}(u-z)\,d\nu_j\biggr|
&\leq\biggl|\int_{\widehat{K}\setminus U}(u-z)\,d\nu_j\biggr|+\int_{U}|u-z|\,d|\nu_j|\\
&\leq(1-\eps)3\delta+2|\nu_j|(U)\\
&<5\delta,
\end{align*}
and thus
\begin{align*}
|\nu_j(u-z)|
&=\biggl|\int_{\Omega}(u-z)\,d\nu_j\biggr|\leq\biggl|\int_{E}(u-z)\,d\nu_j\biggr|+\biggl|\int_{D}(u-z)\,d\nu_j\biggr|\\
&\leq\int_{E\setminus K_E}|u-z|\,d|\nu_j|+\int_{D\setminus\widehat{K}}|u-z|\,d|\nu_j|+\biggl|\int_{\widehat{K}}(u-z)\,d\nu_j\biggr|\\
&\leq2|\nu_j|(E\setminus K_E)+2|\nu_j|(D\setminus\widehat{K})+5\delta
<9\delta\\
&<\gamma.
\end{align*}

On the other hand, since, for every $t\in\widehat{K}\setminus U$,
\[
|u(t)|\geq1-\eps-|z(t)|>1-2\eps>\eps,
\]
one has $\widehat{K}\setminus U\subset B^u_0$ and thus
\begin{align*}
\mu(B^u_0)
&\geq\mu(\widehat{K}\setminus U)\geq\mu(\widehat{K})-\mu(U)>\mu(\widehat{K})-\delta\\
&=\mu(B)-\mu(B^z_0)-\mu(D\setminus\widehat{K})-\delta>\frac13-2\eps-\delta-\delta>\frac13-4\eps\\
&>2\eps.
\end{align*}
This proves that $u\notin\dfrac12S_1+\dfrac12S_2$.
\end{proof}

In particular, one has the following corollary from Theorem~\ref{thm: C_0(K) with K non-scattered} (and Remark \ref{rem: conv. comb. of slices in C_0(K) reach the sphere}),

\begin{cor}
The Banach space $C[0,1]$ has property (P3) but fails (P2).
\end{cor}

\section*{Acknowledgements}
We thank Trond A.~Abrahamsen and Vegard Lima for helpful communications.

\bibliographystyle{amsplain}
\footnotesize

\begin{thebibliography}{10}

\bibitem{AL}
{T}.~{A}.~{A}brahamsen and {V}.~{L}ima, 
\emph{{R}elatively weakly open convex combinations of slices}, {\tt arXiv:1701.} {\tt 06169v1 [math.FA]} (2017).
%
\bibitem{Aliprantis&}
{C}.~{D}.~{A}liprantis and {K}.~{C}.~{B}order,
\emph{Infinite dimensional analysis}, A hitchhiker's guide,
Springer, Berlin,
2006.
%
\bibitem{bonsall_numerical_1973}
{F}.~{F}.~{B}onsall and {J}.~{D}uncan, \emph{Numerical Ranges II}, London Mathematical Society
Lecture Notes Series 10, Cambridge University Press, New York, 1973.
%
\bibitem{ggms}
{N}.~{G}houssoub, {G}.~{G}odefroy, {B}.~{M}aurey, {W}. {S}chachermayer,
\emph{{S}ome topological and geometrical structures in {B}anach spaces},
Mem.~Amer.~Math.~Soc. {\bf 70} (1987), no.~378, {iv+116}.

\bibitem{HPP}
R.~Haller, K.~Pirk, and M.~P\~{o}ldvere, \emph{Diametral strong diameter two property of Banach spaces is stable under direct sums with 1-norm}, Acta Comment. Univ. Tartu. Math. \textbf{20} (2016), 101--105.

\bibitem{Knowles}
{J}.~{D}.~{K}nowles,
\emph{On the existence of non-atomic measures},
Mathematika \textbf{14} (1967), 62--67.

\bibitem{Luther}
{N}.~{Y}.~{L}uther, \emph{Weak denseness of nonatomic measures on perfect, locally compact spaces}, Pacific J. Math. \textbf{34}
(1970), 453--460.

\bibitem{Pel&Sem}
{A.}~{P}e\l czy\'nski and {Z}.~{S}emadeni,
\emph{Spaces of continuous functions. III. Spaces $C(\Omega )$ for
$\Omega $ without perfect subsets},
Studia Math.\textbf{18}
(1959), 211--222.

\bibitem{Rudin}
{W}.~{R}udin,
\emph{Continuous functions on compact spaces without perfect subsets}, Proc. Amer. Math. Soc. \textbf{8} (1957), 39--42.


%
%
%
%
%
%
%
%
%
%

\end{thebibliography}

{\providecommand{\noopsort}[1]{}}
\providecommand{\bysame}{\leavevmode\hbox to3em{\hrulefill}\thinspace}
\providecommand{\MR}{\relax\ifhmode\unskip\space\fi MR }
\providecommand{\MRhref}[2]{%
  \href{http://www.ams.org/mathscinet-getitem?mr=#1}{#2}
}
\providecommand{\href}[2]{#2}

\end{document}